\documentclass[reqno]{amsart}
\usepackage{amsmath}
\usepackage{mathtools}
\usepackage{amssymb}
\usepackage{amscd}
\usepackage{graphics}
\usepackage{latexsym}
\usepackage{color}
\usepackage{verbatim}
\usepackage{extarrows}
\usepackage{latexsym}
\usepackage{enumerate}
\usepackage{arydshln}
\usepackage{soul}
\usepackage{tikz,pgfplots}
\usepackage{tikz-cd}
\usepackage{ulem} 
\definecolor{darkgreen}{HTML}{3CB50F}
\usetikzlibrary{arrows,calc,backgrounds}

\usepackage{hyperref}
\hypersetup{
    colorlinks=true,
    linkcolor=red, 
    filecolor=red, 
    urlcolor=red,
   citecolor=red }
    
  \usepackage[alphabetic]{amsrefs}

\usepackage[margin=1in]{geometry}
\usepackage[all]{xy}
\theoremstyle{plain}
\newtheorem{theorem}{Theorem}[section]
\newtheorem{thm}[theorem]{Theorem}
\newtheorem{cor}[theorem]{Corollary}
\newtheorem{lem}[theorem]{Lemma}
\newtheorem{prop}[theorem]{Proposition}

\theoremstyle{definition}
\newtheorem{definition}[theorem]{Definition}

\newtheorem{conj}[theorem]{Conjecture}
\newtheorem{remark}[theorem]{Remark}

\newtheorem{rmk}[theorem]{Remark}

\theoremstyle{theorem}
\newtheorem*{thmSN}{Theorem $\star$}
\newtheorem{thmIntro}{Theorem}
\newtheorem{defIntro}{Definition}

\newtheorem{coro1}{Corollary}
\newtheorem*{coro2}{Corollary 2}
\newtheorem{question}{Question}

\theoremstyle{remark}

\newcommand{\PP}{\mathbb{P}}

\newcommand{\OO}{\mathcal O}

\def\PP{{\mathbb P}}

\title{The Noether-Lefschetz locus of surfaces in $\PP^3$ formed by determinantal surfaces}
\author{}

\begin{document}
\author{Manuel Leal}
\author{C\'esar Lozano Huerta}
\author{Montserrat Vite}

\address{Currently: Harvard University\\
Department of Mathematics \\
Oxford 1, Cambridge, MA, USA.}
\email{lozano@math.harvard.edu}

\address{Permanent: Universidad Nacional Aut\'onoma de M\'exico\\
Instituto de Matem\'aticas, Unidad Oaxaca \\ Mex.}
\email{lozano@im.unam.mx}

\address{Universidad Nacional Aut\'onoma de M\'exico\\
Instituto de Matem\'aticas, Unidad Oaxaca \\
Oaxaca, Mex.}
\email{maz.leal.camacho@gmail.com}
\email{lilia.vite@matem.unam.mx} 


\maketitle

\begin{abstract}
 We compute the dimension of certain components of the family of smooth determinantal degree $d$ surfaces in $\PP^3$, and show that each of them is the closure of a component of the Noether-Lefschetz locus $NL(d)$. Our computations exhibit that smooth determinantal surfaces in $\PP^3$ of degree 4 form a divisor in $|\mathcal{O}_{\PP^3}(4)|$ with 5 irreducible components. We will compute the degrees of each of these components: $320,2508,136512,38475$ and $320112$.
\end{abstract}

\bigskip

\section*{Introduction}

\noindent
A smooth complex cubic surface $X\subset \PP^3$ is known to be linear determinantal since Grassmann in 1855 \cite{Grassman}. Any of its determinantal expressions provides a birational morphism $X\rightarrow \PP^2$, which is key in studying $X$ intrinsically. If $d>3$, a general degree $d$ surface in $\PP^3$ is not linear determinantal. In this paper, we study surfaces that are defined by the determinant of a matrix with polynomial entries, where degrees of these entries satisfy only minor constrains (Definitions \ref{def::admissiblePair} and \ref{def::detAB}). And we do so, first by manipulating this matrix and then intrinsically. Our analysis enhances Beauville's characterization of linear determinantal hypersurfaces \cite{Beau}, and builds on A. F. Lopez \cite{Angelo} insights on the Picard group of surfaces in $\PP^3$.

\medskip\noindent
Our first goal is to answer the question below when the degrees of the entries are arbitrary. 
In other words, let us consider that a \textit{determinantal surface} is the zero locus of the determinant of a matrix whose entries are homogeneous, not necessarily linear, polynomials. We thus ask:

\begin{question}\label{qtn::question1}
    What is the dimension of the family of determinantal surfaces of degree $d$ in $\mathbb{P}^3$?
\end{question}

\medskip\noindent
Our main result, Theorem \ref{thm::mainTheorem}, answers this question under mild hypothesis. Of course, if we are to write polynomials of degree $d$ as determinants of square matrices, then we have to choose the size of the matrix and the degrees of its entries. If $d>3$, we will show that these choices render the locus of degree $d$ determinantal surfaces reducible, with irreducible components of different dimensions. By keeping track of these two choices Theorem \ref{thm::mainTheorem} computes the dimension of each of such components. For example, if $d=5$, then there is a component of codimension $2$ and another of codimension $3$, and there are 6 others, each of codimension $4$ in $|\mathcal{O}_{\PP^3}(5)|$. We find this distribution of codimensions revealing (see Table \ref{T2}); let us comment on it.

\medskip\noindent
Let $NL(d)\subset|\OO_{\PP^3}(d)|$ be the Noether-Lefschetz locus; that is, $NL(d)$ parametrizes smooth degree $d$ surfaces whose Picard group is not generated by the hyperplane class. A smooth complex determinantal surface never has Picard rank 1. Therefore, the family of these surfaces is contained in $NL(d)$, which has infinitely many irreducible components. Each of these components has a rich and intricate geometry and general results about $NL(d)$ are scarce. But one of those is the following.

\begin{thmSN}[\cite{Green, Voisin, GH2}] \label{thmSN}
    If $\Sigma$ is an irreducible component of $NL(d)$ then
    $$d-3\leq \mathrm{codim} \ \! \Sigma \leq \binom{d-1}{3}.$$      \end{thmSN}
    
\noindent
The upper bound in Theorem $\star$ can be seen as follows. In a smooth degree $d$ surface $S\subset \PP^3_{\mathbb{C}}$, an integral class $\gamma\in H^{2}(S;\mathbb{Z})$ to be algebraic imposes $h^{2}(\mathcal{O}_S)=\binom{d-1}{3}$ conditions  \cite[p. 179]{GH2}, \cite[Th. 3.3.11]{sernesi}.  Of course, these conditions may be dependent. A component $\Sigma$ for which these conditions are independent is called \textit{general}; otherwise, it is called \textit{special}.

\medskip\noindent
The inequality $d-3\le \mathrm{codim} \ \Sigma$ is more subtle and holds by a result of M. Green \cite{Green}, and independently C. Voisin \cite{Voisin}. In fact, for $d\geq5$ there is a unique special component of codimension $d-3$, which is special, formed by surfaces of degree $d$ that contain a line. This type of surfaces are determinantal defined by $2\times 2$ matrices with entries of degree $1$ and $d-1$. C. Voisin further proved that surfaces of degree $d$ that contain a conic form the unique irreducible component of $NL(d)$ of codimension $2d-7$ \cite{Voisin2}. This component also parametrizes determinantal surfaces (Remark \ref{rmk::lineConic}), and is special for $d\ge 5$. The following Corollary asserts that these are particular cases of a general phenomenon.

\medskip
\begin{coro1}\label{candites}
    Every irreducible component of the family of degree $d$ determinantal surfaces contains as an open set a component of the Noether-Lefschetz locus $NL(d)$ and the general surface in it has Picard rank 2.
\end{coro1}

\medskip\noindent
Are the components formed by determinantal surfaces \textit{special} or \textit{general}? 
 There have been many papers on identifying and describing components of the Noether-Lefschetz loci in the 1980s, 1990s and early 2000s. Many of these papers used the philosophy that components parametrizing surfaces containing low degree curves should have large dimension, and are therefore special. 

\medskip\noindent
We found that smooth determinantal surfaces typically contain (ACM) curves of low degrees, which explains why these surfaces often yield \textit{special} components (Corollary 2). Indeed, Theorem \ref{thm::mainTheorem} tells us the dimension of such components, and reveals that determinantal surfaces provide mostly \textit{special} components of $NL(d)$ (see Table \ref{T2}); many whose existence had not been yet predicted as far as we know (Remark \ref{rmkCA}).

\medskip\noindent
Summarizing, we show that the determinantal surfaces with fixed numerical invariants fill up a component of the Noether-Lefschetz locus, and Theorem \ref{thm::mainTheorem} calculates its dimension.

\medskip\noindent
Let us now state Theorem \ref{thm::mainTheorem} precisely before stating our second result. We work over $\mathbb{C}$ throughout. Let $X\subset\PP^3$ be a surface defined by the determinant of a matrix
\begin{equation}\label{eq::determinantalSurface}
        S=\left(
        \begin{array}{cccc}
            m_{11} & m_{12} &  \ldots & m_{1t}\\
            m_{21} & m_{22} & \ldots & m_{2t}\\
            \vdots & \vdots & \ddots & \vdots\\
            m_{t1} & m_{t2} & \ldots & m_{tt}
        \end{array}
        \right), 
    \end{equation}
with homogeneous polynomials $m_{i,j}\in\mathbb{C}[x, y, z, w]$. If the degrees of the $m_{ij}$ are arbitrary, then $det(S)$ may not be homogeneous. For $det(S)$ to be an homogeneous polynomial of degree $d$, and to keep track of the degree of each $m_{ij}$, we make the following definition.

\begin{defIntro}\label{def::admissiblePair}
    Consider two sequences of integers $a=(a_1,\ldots,a_t), b=(b_1,\ldots,b_t)\in\mathbb{Z}^t$, with $t\geq2$, and
    \begin{enumerate}[(i)]
        \item $a_1\leq a_2\leq\ldots\leq a_t$,
        \item $b_1\leq b_2\leq\ldots\leq b_t$,
        \item $a_i<b_j$ for all $1\leq i,j\leq t$.
    \end{enumerate}
    We say that $(a,b)$ is an \textit{admissible pair} of degree $d:=\sum_{i=1}^tb_i-a_i$ and length $t$. Two admissible pairs $(a,b)$ and $(a',b')$ are \textit{equivalent} if there exists some $k\in\mathbb{Z}$ such that $a_i'=a_i+k$ and $b_i' = b_i+k$ for all $1\leq i\leq t$.
\end{defIntro}

\begin{defIntro}\label{def::detAB}
    Let $(a, b)$ be an admissible pair of degree $d$ and length $t$. A surface $X=\{F=0\}\subset\PP^3$ is called a \textit{determinantal surface} of type $(a,b)$ if $F=det(S)$ for some matrix $S$ as in (\ref{eq::determinantalSurface}), where $m_{ij}$ is an homogeneous polynomial of degree $b_j-a_i$. The closure of the family of all determinantal surfaces of type $(a,b)$ in $\PP^3$ will be denoted by
    $$det(a, b):=\overline{\{X=\{det(S)=0\}\ : \ det(S)\not\equiv 0\}}\subset |\OO_{\PP^3}(d)|.$$
\end{defIntro}

\medskip\noindent
    A determinantal surface can be defined by a pair $(a,b)$ in Definition \ref{def::admissiblePair} satisfying (i), (ii) and a less restrictive condition than (iii). For example, Proposition \ref{prop1} below holds under the milder condition $a_{i+1}\leq b_{i}$. However, our results rely on Theorem \ref{Angelo}, which requires a curve defined in Section \ref{section::2}, denoted $C_{-}$, to be smooth. This smoothness is guaranteed if condition $(iii)$ is satisfied (Proposition \ref{prop::smoothCurve}). Remark \ref{refereeExample} exhibits that Corollary \ref{candites} can fail, if condition $(iii)$ in Definition \ref{def::admissiblePair} is modified. Also, if $a_i=b_j$ for some $i,j$, (and $a_i<b_j$ otherwise), then $a_i$ and $b_j$ can be removed from the defining pair, producing a new pair $(a',b')$ which satisfies Definition \ref{def::admissiblePair}. This reduces the analysis of $det(a,b)$ to that of $det(a',b')$ using Definition \ref{def::admissiblePair}.

\medskip\noindent
Observe that $b_j-a_i=b_j'-a_i'$ for equivalent pairs $(a,b)$ and $(a',b')$; therefore, $det(a,b)$ depends only on $(a,b)$ up to equivalence. In particular, we can assume that $a_1>d$.

\medskip\noindent
Using these definitions, Question \ref{qtn::question1} asks what is the dimension of each family $det(a,b)$.   
The following theorem answers this and is our main result.

\begin{thmIntro}\label{thm::mainTheorem}
    Let $(a,b)$ be an admissible pair of degree $d$ such that $a_1>d$. Set $a_0:=d$. If we denote
    $$A:=\bigoplus_{i=0}^t\OO_{\PP^3}(-a_i),  \quad  B:=\bigoplus_{j=1}^t\OO_{\PP^3}(-b_j)$$
    then $det(a,b)$ is an irreducible variety of dimension
    \begin{equation}\label{eq::mainTheorem}
        \dim det(a,b)=2+hom(B,A)-hom(A,A)-hom(B,B)-\binom{d-1}{3}-g_C+(d-4)d_C
    \end{equation}
    where $d_C,g_C$ are the degree and genus of a curve $C$ with ideal sheaf $\mathcal{I}_C$ given by a resolution $B\hookrightarrow A\twoheadrightarrow  \mathcal{I}_C$.
\end{thmIntro}

\medskip\noindent
The expression (\ref{eq::mainTheorem}) can be written explicitly in terms of $(a,b)$.
For example, $d_C$ and $g_C$ 
are written in (\ref{eq::degreeAndGenus}). Macaulay2 code for computing (\ref{eq::mainTheorem}) in terms of $(a,b)$ 
can be found in \cite{MV}. 
Theorem \ref{thm::mainTheorem} recovers the well-known dimension of the family of linear determinantal surfaces (Corollary \ref{cor::linearEntries}), and also the minimal codimension components of $NL(d)$ (see Remark \ref{rmk::lineConic}).

\medskip\noindent
The proof of Theorem \ref{thm::mainTheorem} carries out\footnote{In fact, M. Noether aims to compute the dimension of families of curves in $\PP^3$ by estimating codim $\Sigma$ \cite[\S 11]{Noether}.} an idea by M. Noether, and relies on the study of the Hilbert scheme of curves in $\PP^3$. For one thing, given a smooth determinantal surface $X\in  |\mathcal{O}_{\PP^3}(d)|$, we construct a curve $C \subset X$, of genus $g_C$ and degree $d_C$, with good properties (Lemma \ref{lem::goodCohomologicalProperties}): an ACM curve. These curves have been thoroughly studied.  F. Gaeta (in response to a question by F. Severi) characterized ACM curves as those linked to a complete intersection \cite{Gaeta}. C. Peskine and L. Szpiro revisited Gaeta's results in their \textit{th\'eor\`eme de liaison} \cite{Peskine}, and building on it, G. Ellingsrud showed that the Hilbert scheme at an ACM curve is smooth \cite[Th. 2]{ellingsrud}. Using Ellingsrud's result, we compute the dimension of the Hilbert scheme of such curves and answer Question \ref{qtn::question1} via an incidence correspondence.

\medskip\noindent
 We highlight that the Hilbert scheme of the aforementioned $C$ often fails to have the expected dimension $4d_C$. The discrepancy between $4d_C$ and the actual dimension is $h^1(N_{C/\PP^3})$, and we will show that this discrepancy and the speciality of $\mathcal{O}_C(d)$ determine whether $\Sigma$, the Noether-Lefschetz component that contains $X$, is \textit{special} or \textit{general}. Thus, in this context, the following is a refinement of Theorem $\star$.

\begin{coro2}\label{cor::sernesi}
    Let $\Sigma\subset NL(d)$ be a component formed by smooth determinantal surfaces. Then, $\Sigma$ is generically smooth with
    $$codim \ \Sigma  = \binom{d-1}{3} -\kappa,$$
    where $\kappa := h^1(N_{C/\PP^3})- h^1(\mathcal{O}_C(d))$, and this number is always non-negative.
\end{coro2}

\medskip\noindent
The case of quartic surfaces is curious. A general quartic surface is not determinantal, and when it is, distinct choices for an admissible pair $(a,b)$ give rise to 5 components of $NL(4)\subset |\mathcal{O}_{\mathbb{P}^3}(4)|$; all of codimension 1. The determinantal expression for the surfaces in these components allows us to study each of these divisors intrinsically, via lattices of rank $2$ of $K3$ surfaces. Here is what we can prove. 
\begin{thmIntro}\label{thm::quarticDivisors}
    The family of determinantal quartic surfaces consists of 5 prime divisors $\mathcal{F}_1,\ldots,\mathcal{F}_5\subset|\OO_{\PP^3}(4)|$ depending on the choice of an admissible pair $(a,b)$ of degree 4. Each of these divisors has degree
    \begin{align*}
        deg(\mathcal{F}_1)&=320112,\\
        deg(\mathcal{F}_2)&=136512,\\
        deg(\mathcal{F}_3)&=38475 ,\\
        deg(\mathcal{F}_4)&=320,\\
        deg(\mathcal{F}_5)&=2508.
    \end{align*}
\end{thmIntro}
\noindent
This result tells us that identifying a determinantal surface based solely on the coefficients of its defining equation is likely a challenging problem.

\section*{Organization of the paper}

\noindent
Section 1 has the preliminary results and definitions necessary to prove Theorem \ref{thm::mainTheorem}, and Section 2 is where its proof is presented. Section 3 studies some distinguished families $det(a,b)$. In Section 4, we investigate determinantal quartic surfaces using the Noether-Lefschetz theory of K3 surfaces and prove Theorem \ref{thm::quarticDivisors}. We work over $\mathbb{C}$ throughout. 

\section*{acknowledgments}
\noindent We thank Lara Bossinger, Antonio Laface and Bernd Sturmfels for useful conversations. Special thanks to Angelo Felice Lopez, Joe Harris, Anand Patel and Edoardo Sernesi for their generosity in sharing their insights with us, and for pointing out important literature on the subject. CLH is currently a CONAHCYT Research Fellow in Mathematics, project No. 1036, and has been partially supported by CONAHCYT, grant ``Estancias sab\'aticas en el extranjero 2023'' No. I1200/311/2023. Also, he thanks the department of Mathematics at Harvard for the generous hospitality during his sabbatical visit. During the preparation of this note MV and ML were funded by CONAHCYT under the program of ``Becas de Posgrado''. We also thank the referees for their careful review and valuable suggestions.

\section{Determinantal surfaces in $\PP^3$}\label{section::2}
\noindent
This section collects the preliminary results and definitions needed in the proof of Theorem \ref{thm::mainTheorem}. We start by analyzing the irreducibility of the families $det(a,b)$ (see Definition \ref{def::detAB}) and showing that the subset of $det(a, b)$ parametrizing smooth surfaces is not empty. We will denote this subset by $det(a, b)^{sm}$.

\begin{prop}\label{prop1}
    Let $(a,b)$ be an admissible pair. Then $det(a,b)$ is irreducible and the general element in $det(a,b)$ is a smooth surface.
\end{prop}
\begin{proof}
Let us consider
    $$V:=\bigoplus_{1\leq i,j\leq t}H^0(\PP^3,\OO_{\PP^3}(b_j-a_i))$$
    as the vector space of matrices $S$ with entries homogeneous polynomials of degree $b_j-a_i$. Then, by taking the determinant one gets a rational map
    $$det:V\dashrightarrow|\OO_{\PP^3}(d)|$$
    dominating $det(a,b)$, which proves the irreducibility. The second part follows from the fact that smoothness is an open condition. Thus, it is enough to exhibit a matrix realizing a smooth surface. One such matrix is
    \begin{align*}
        S:=\left(
        \begin{array}{ccccc}
            f_1 & 0 & \ldots & 0 & g_t\\
            g_1 & f_2 & \ldots & 0 & 0\\
            0 & g_2 & \ldots & 0 & 0\\
            \vdots & \vdots & \ddots & \vdots & \vdots\\
            0 & 0 & \ldots & f_{t-1} & 0\\
            0 & 0 & \ldots & g_{t-1} & f_t
        \end{array}
        \right)
    \end{align*}
    where $\prod_i f_i=x^d+y^d$ and $(-1)^t\prod_i g_i=z^d+w^d$. In fact, in this case $det(S)=x^d+y^d+z^d+w^d$ determines the Fermat surface of degree $d$, which is smooth.
\end{proof}

\noindent
The proof above exemplifies that it is possible to produce surfaces with high Picard rank by specializing the defining matrix $S$. It would be interesting to know how this specialization might work in a systematic way.

\subsection{ACM curves contained in determinantal surfaces}
\noindent
Given a determinantal surface $X\subset \PP^3$, there are curves contained in $X$ with good cohomological properties. By computing the dimension of the maximal family of these curves, we will calculate the dimension of the family of determinantal surfaces. We describe such curves next.

\begin{definition}\label{def1}
    A curve $C\subset\PP^3$ is called \textit{arithmetically Cohen-Macaulay}, or ACM, if its ideal sheaf admits a minimal free resolution of the form
    \begin{align}\label{eq::ACMresolution}
        0\to\bigoplus_{j=1}^t\OO_{\PP^3}(-b_j)\xrightarrow{M}\bigoplus_{i=0}^t\OO_{\PP^3}(-a_i)\to\mathcal{I}_C\to0.
    \end{align}
\end{definition}

\medskip\noindent
Given a curve $C$ as in (\ref{eq::ACMresolution}), we denote its degree by $d_C$, and by $g_C$ its genus. These numbers only depend on the graded Betti numbers $a_i$ and $b_j$, and one can explicitly compute them:
\begin{equation}\label{eq::degreeAndGenus}
    d_C=\frac{1}{2}\left(\sum_{j=1}^tb_j^2-\sum_{i=0}^ta_i^2\right), \quad \quad \quad
       g_C=1+\frac{1}{6}\left(\sum_{j=1}^tb_j^3-\sum_{i=0}^ta_i^3\right)-2 d_C.
\end{equation}

\medskip\noindent
We focus on the family of smooth ACM curves in the Hilbert scheme of curves of degree $d_C$ and genus $g_C$, with graded Betti numbers as in (\ref{eq::ACMresolution}). We denote this locus by $\mathcal{H}_{a,b}$ and its dimension can be computed by the following result. See also \cite[Prop. 3.3]{BS}.

\begin{thm}\cite[Th. 2]{ellingsrud}\label{thm::timFormula}.
    Let $a_0\leq\ldots\leq a_t$ and $b_1\leq\ldots\leq b_t$ be integers such that $\sum_{j=1}^tb_j=\sum_{i=0}^ta_i.$ Denote
    \begin{align*}
        A=\bigoplus_{i=0}^t\OO_{\PP^3}(-a_i), \quad \quad B=\bigoplus_{j=1}^t\OO_{\PP^3}(-b_j)
    \end{align*}
    and assume that there exist curves $C$ with free resolution (\ref{eq::ACMresolution}). Then the family $\mathcal{H}_{a,b}$ of such curves has dimension
    $$\mbox{dim}\ \mathcal{H}_{a,b}=hom(B,A)+hom(A,B)-hom(A,A)-hom(B,B)+1,$$
    where $hom(X,Y):=h^0(\PP^3, Hom_{\OO_{\PP^3}}(X,Y)).$
\end{thm}

\medskip\noindent
    Throughout the paper we will further assume that $a_i < b_j$ for all $i,j$. This ensures that there exist curves with resolution \eqref{eq::ACMresolution}, and this resolution is minimal.

\begin{remark}\label{rmk::expressionHomAB}
    The dimension of $\mathcal{H}_{a,b}$ above can be written explicitly in terms of the $a_i$'s and $b_j$'s. For example:
    $$hom(B,A)=\sum_{\substack{0\leq i\leq t\\ 1\leq j\leq t}}\binom{b_j-a_i+3}{3}.$$
\end{remark}

\noindent
We now state a particular case of \cite[Th. 2]{ellingsrud}, adapted for our purposes.
\begin{prop}\label{prop::Migliore}
    Suppose that, in (\ref{eq::ACMresolution}), $a_i<b_j$ for all $i,j$. Then the family $\mathcal{H}_{a,b}$ of ACM curves with minimal free resolution (\ref{eq::ACMresolution}) is an irreducible, smooth open set of the Hilbert scheme $Hilb_{d_C,g_C}^{\PP^3}$.
\end{prop}

\medskip\noindent
The minimal free resolution of a curve $C$ in Theorem \ref{thm::timFormula} determines a matrix
\begin{equation*}
    M=\left(\begin{array}{cccc}
        m_{01} & m_{02} & \ldots & m_{0t}\\
        m_{11} & m_{12} & \ldots & m_{1t}\\
        \vdots & \vdots & \ddots & \vdots\\
        m_{t1} & m_{t2} & \ldots & m_{tt}
    \end{array}\right),
\end{equation*}
where $m_{ij}$ is an homogeneous polynomial of degree $b_j-a_i$. If $a_i<b_j$ for all $i,j$, as in Proposition \ref{prop::Migliore}, then ACM curves are smooth.

\begin{prop}\cite[Th. 6.2]{Peskine  }\label{prop::smoothCurve}
    If $C$ is a curve defined by a minimal free resolution as in (\ref{eq::ACMresolution}) where $a_i<b_j$ for all $i,j$ and $M$ is general, then $C$ is smooth and irreducible.
\end{prop}

\medskip\noindent

\begin{remark}\label{rmk::detSurfACM}
    Given a curve $C$ with minimal resolution (\ref{eq::ACMresolution}), the $t\times t$-minors $F_i$ obtained by removing the $i$-th row of $M$ form a set of minimal generators for the ideal $\mathcal{I}_C$. In particular, $\{F_i=0\}$ is a determinantal surface of type $a=(a_0,\ldots,\hat{a_i},\ldots,a_t), b=(b_1,\ldots,b_t)$ containing $C$. Thus, ACM curves are naturally contained in determinantal surfaces.
\end{remark}

\medskip\noindent
We now define two types of ACM curves that a determinantal surface always contains. We denote them by $C_+$ and $C_-$.

\medskip\noindent
First consider a general surface $X\in det(a,b)$ given by a matrix $S$ as in (\ref{eq::determinantalSurface}). Assume that $a_1>d$ and set $a_0:=d$. Let us consider a row
$$R=(m_{01}, m_{02}, \ldots,m_{0t}),$$
where $m_{0j}$ stands for a general homogeneous polynomial of degree $deg(m_{0j})=b_j-a_0$. Then the minors of the matrix
$$S_+:=\left(
\begin{array}{c}
R\\
S
\end{array}
\right)$$
define an ACM curve $C_+\subset X$ with minimal free resolution (\ref{eq::ACMresolution}). We say that $C_+$ is a curve in $X$ \textit{obtained by adding a row}.

\medskip\noindent
Similarly, consider the matrix
$$S_-:=\left(
\begin{array}{cccc}
    m_{11} & m_{12} & \ldots & m_{1,t-1}\\
    m_{21} & m_{22} & \ldots & m_{2,t-1}\\
    \vdots & \vdots & \ddots & \vdots\\
    m_{t1} & m_{t2} & \ldots & m_{t,t-1}
\end{array}
\right)$$
obtained by removing the last column of $S$. The $(t-1)\times(t-1)$-minors of $S_-$ define an ACM curve $C_-\subset X$ with minimal free resolution
\begin{equation}\label{eq::ACMresolution2}
    0\to\bigoplus_{j=1}^{t-1}\OO_{\PP^3}(-d+b_t-b_j)\to
    \bigoplus_{i=1}^t\OO_{\PP^3}(-d+b_t-a_i)\to\mathcal{I}_{C_-}\to0.
\end{equation}
We say that $C_-$ is a curve in $X$ \textit{obtained by removing a column}. 

\medskip\noindent
The general determinantal surface always contains curves obtained either by adding a row or removing a column, and the general such curve is smooth and irreducible by Proposition \ref{prop::smoothCurve}. Let us now prove that if a surface contains a curve like these, then it is determinantal.

\begin{prop}(Determinantal criterion)\label{prop::keyObservation}
    Let $(a, b)$ be an admissible pair of degree $d$ with $a_1>d$, and let $C_+, C_-$ be curves with minimal free resolution (\ref{eq::ACMresolution}) and (\ref{eq::ACMresolution2}) respectively. Then any degree $d$ surface $X=\{F=0\}$ containing $C_+$ or $C_-$ is determinantal of type $(a,b)$. Moreover, in the first case $X$ is the unique surface of degree $d$ containing $C_+$. In particular, $F$ is a minimal generator of the ideal $\mathcal{I}_{C_+}$.
\end{prop}
\begin{proof}
     When $C_+$ has resolution (\ref{eq::ACMresolution}), we saw that $X$ is determinantal in Remark \ref{rmk::detSurfACM}. The claim that $X$ is the only degree surface containing $C_+$ follows from the assumption $a_1>d$.
     
     For the other case, notice that $C_-\subset X$ is equivalent to having an expression
    $$F=\sum_{i=1}^t(-1)^im_{it}F_i$$
    where the $F_i$ are the minimal generators of $\mathcal{I}_{C_-}$; that is, the minors of the matrix $S_-$ defining $C_-$. If we define a matrix $S$ by completing $S_-$ with these $m_{it}$, expanding the determinant of $S$ with respect to the last column gives the expression $det(S)=\pm F$, proving $X\in det(a,b)$.
\end{proof}

\medskip\noindent
Now we combine Proposition \ref{prop::keyObservation} with the following result in order to compute the Picard rank of a general determinantal surface.
\begin{thm}\cite[Cor. II.3.8]{Angelo} \label{Angelo}
    Let $C\subset\PP^3$ be a smooth curve and let $m(C)$ be the maximum degree among a minimal set of generators of $\mathcal{I}_C$. Let
    $$d\geq\max\{4,m(C)+1\}.$$
 Then the general surface of degree $d$ containing $C$ has Picard group isomorphic to $\mathbb{Z}^2$, with generators $\OO_X(C)$ and $\mathcal{O}_X(1)$.
\end{thm}

\begin{cor}\label{cor::detInNL}
    If $d\geq4$, then the general $X\in det(a,b)$ has Picard group isomorphic to $\mathbb{Z}^2$. In particular, every degree $d$ smooth determinantal surface is contained in $NL(d)$.
\end{cor}
\begin{proof}
    Let $(a,b)$ be an admissible pair and consider a general curve $C_-$ with minimal free resolution (\ref{eq::ACMresolution2}). Since every minimal generator of $\mathcal{I}_{C_-}$ is of degree strictly less than $d$, it follows from Theorem \ref{Angelo} that the general degree $d$ surface $X$ containing $C_-$ has Picard group $\mathbb{Z}^2$. By Proposition \ref{prop::keyObservation}, the general $X\in det(a,b)$ is obtained this way.
\end{proof}

\medskip\noindent
Fix a degree $d$ admissible pair with $a_1>d$. The previous argument proves that the Picard group of the general $X\in det(a,b)$ is generated by $\OO_X(C_-)$ and $\OO_X(1)$, where $C_-\subset X$ is a curve obtained by removing a column. We could also consider a curve $C_+\subset X$ obtained by adding a row, and it turns out that $\OO_X(C_+)$ and $\OO_X(1)$ is again a basis for $Pic(X)$.

\begin{prop}
    If $d\geq4$, $X\in det(a,b)$ is general and $C_+\subset X$ is obtained by adding a row, then $Pic(X)$ is generated by $C_+$ and the hyperplane class $H$.
\end{prop}
\begin{proof}
    Consider the matrix $S_+$ defining the curve $C_+$. For each $u\in\mathbb{C}$, define
    \begin{equation*}
        S_u:=\left(
        \begin{array}{ccccc}
            u\cdot m_{01} & u\cdot m_{02} & \ldots & u\cdot m_{0,t-1} & m_{0t}\\
            m_{11} & m_{12} & \ldots & m_{1,t-1} & m_{1t}\\
            \vdots & \vdots & \ddots & \vdots & \vdots\\
            m_{t1} & m_{t2} &\ldots & m_{t,t-1} & m_{tt}
        \end{array}
        \right)
    \end{equation*}
    and let $C_u$ be the scheme defined by the $t\times t$ minors of $S_u$. Observe that $C_1=C_+$ and $C_0 = C_-\cup\{m_{0t}=0\}$, and in addition this construction makes $C_0$, $C_1$ rationally equivalent in $X$. Since $\deg(m_{0t})=b_t-d$, then we have that $C_0=C_-+(b_t-d)H$, as divisor classes. Since $H^1(X, \OO_X)=0$, therefore, $C_+$ and $C_-+(b_t-d)H$ are linearly equivalent. This completes the proof.
\end{proof}

\begin{remark}\label{refereeExample}
If we replace in Definition \ref{def::admissiblePair} the condition $(iii)$ with $ a_{i+1}\le b_i$, then Corollary \ref{candites} may fail. For example, consider the pair $a=(6,6,8)$, $b=(7,9,9)$ and let $X$ be a generic element in $det(a,b)$. In this case, $C_{-}$ is a reducible curve (the union of a plane quartic and a line), and consequently our arguments break down. Indeed, $X$ has Picard rank at least 3 and $det(a,b)$ fails to be a component of the Noether-Lefschetz $NL(5)$. In fact, $det(a,b)$ parametrizes generically quintics that contain two skew lines.
\end{remark}

\section{Proof of Theorem \ref{thm::mainTheorem}}
\noindent
In this section, we answer Question \ref{qtn::question1} by proving Theorem \ref{thm::mainTheorem}. From now on, we will assume that $(a,b)$ is an admissible pair of degree $d$ and length $t$ with $a_1>d$. Imposing this last inequality is not restrictive and will simplify our computations. Set $a_0:=d$. Take a general element $X=\{F=0\}\in det(a,b)$ with $F=det(S)$ for a matrix $S$ as in (\ref{eq::determinantalSurface}). Let us consider a $1\times t$ matrix
\begin{align*}
    R=\left(
        \begin{array}{cccc}
            m_{01} & m_{02} & \ldots & m_{0t}
        \end{array}
    \right)
\end{align*}
and the curve $C\subset X$ obtained by adding the row $R$. We denote by $d_C$ and $g_C$ the degree and genus of $C$, respectively. In order to compute the dimension of $det(a,b)$ we need the following lemma which exhibits the cohomological properties of $C$.
\begin{lem}\label{lem::goodCohomologicalProperties}
    If $X$ and $C$ are as above and smooth, then $H^1(X,\mathcal{O}_{X}(C))=H^2(X,\mathcal{O}_{X}(C))=0$.
\end{lem}
\begin{proof}
    Let $H$ denote the hyperplane class of  $X$. Then the canonical class of $X$ is $K_X=(d-4)H$ and $H^2(X,\mathcal{O}_{X}(C))\cong H^0(X,\mathcal{O}_{X}((d-4)H-C))$ embeds into $H^0(X,\mathcal{O}_{X}((d-4)H)).$ Since the restriction map
    $$H^0(\PP^3,\OO_{\PP^3}(d-4))\to H^0(X,\mathcal{O}_{X}((d-4)H))$$
    is surjective, every element in $H^0(X,\mathcal{O}_{X}((d-4)H-C))$ comes from a degree $d-4$ form vanishing on $C$. Given that $d$ is the minimum degree of a surface containing $C$, then $H^2(X,\mathcal{O}_{X}(C))\cong H^0(X,\mathcal{O}_{X}((d-4)H-C))=0$.

    Now we turn our attention to $H^1(X,\mathcal{O}_{X}(C))$. Consider the short exact sequence
    $$0\to\OO_X\to\OO_X(C)\to\OO_C(C)\to0.$$
    Since $H^1(X,\OO_X)=H^2(X,\mathcal{O}_{X}(C))=0$ we obtain an exact sequence
    $$0\to H^1(X,\mathcal{O}_{X}(C))\to H^1(C,\OO_C(C))\to H^2(X,\OO_X)\to 0.$$
    The adjoint linear series and Serre duality imply
    $$H^2(X,\OO_X)\cong H^0(\PP^3,\OO_{\PP^3}(d-4)).$$
    On the other hand, since the canonical sheaf of $C$ is equal to $\omega_C=\OO_C(K_X+C)$, then it follows from Riemann-Roch that
    $$H^1(C,\OO_C(C))\cong H^0(C,\OO_C(K_X))=H^0(C,\mathcal{O}_{C}((d-4)H)).$$
    Since $C$ is an ACM curve, it follows from the short exact sequence
    $$0\to\mathcal{I}_C(d-4)\to\OO_{\PP^3}(d-4)\to\OO_C((d-4)H)\to0$$ that $H^0(C,\mathcal{O}_{C}((d-4)H))\cong H^0(\PP^3,\OO_{\PP^3}(d-4))$ as $H^i(\PP^3,\mathcal{I}_C(d-4))=0$ for $i=0,1$. Therefore, $H^1(X,\mathcal{O}_{X}(C))=0$.
\end{proof}
\begin{rmk}\label{cor::formulaH0}
Let $C\subset X$ be as in the previous Lemma. Since $H^i(X,\mathcal{O}_{X}(C))=0$ for $i=1,2$, then Riemann-Roch yields
    $$h^0(X,\mathcal{O}_{X}(C))=\binom{d-1}{3}+g_C-(d-4)d_C.$$
\end{rmk}

\medskip\noindent
The previous computations allow us to calculate the dimension of $det(a,b)$ next, which answers Question \ref{qtn::question1}. 
\begin{proof}[Proof of Theorem \ref{thm::mainTheorem}]
    Let us denote by $Hilb_C^{\PP^3}$ (resp. $Hilb_C^X$) the Hilbert scheme of curves of degree $d_C$ and arithmetic genus $g_C$ inside $\PP^3$ (resp. $X$). We further denote by $\mathcal{H}_{a,b}\subset Hilb_C^{\PP^3}$ the locus of ACM curves with minimal free resolution (\ref{eq::ACMresolution}). Since we made the convention that $a_1>d$, we obtain a dominant rational map
    $$\mathcal{H}_{a,b}\dashrightarrow det(a,b)$$
    sending a curve in $C\in\mathcal{H}_{a,b}$ to the unique degree $d$ surface $X\in det(a,b)$ containing it. We denote by $\mathcal{H}_C^X$ the fiber of this map over a general $X$. Below, Lemma \ref{lem::technicalLemma} proves that $\mathcal{H}_C^X$ has dimension equal to $\dim|\OO_X(C)|$. Therefore, we get that $\dim det(a,b)=\dim\mathcal{H}_{a,b}-\dim\mathcal{H}_C^X$. Theorem \ref{thm::mainTheorem} now follows from this equation, Theorem \ref{thm::timFormula} and Remark \ref{cor::formulaH0}.
\end{proof}

\begin{lem}\label{lem::technicalLemma}
    If $X\in det(a,b)$ is smooth, every irreducible component of $\mathcal{H}_C^X$ is generically smooth of dimension $\dim|\OO_X(C)|$.
\end{lem}
\begin{proof}
    For any (not necessarily smooth) curve $C\subset X$ with minimal free resolution (\ref{eq::ACMresolution}), we first compute the dimension of the tangent space of $Hilb_C^X$ at $[C]$. Since $N_{C/X}=\mathcal{O}_C(C)$, consider the short exact sequence
    $$0\to\OO_X\to\OO_X(C)\to N_{C/X}\to 0.$$
The long exact sequence in cohomology, together with $H^1(X,\OO_X)=0$, proves that the restriction map
    $$H^0(X,\OO_X(C))\to H^0(N_{C/X})$$
    is surjective and has kernel of dimension $h^0(X,\OO_X)=1$. This proves that the tangent space to the Hilbert scheme $Hilb_C^X$, at $[C]$, has dimension $\dim |\OO_X(C)|$. 
    
On the other hand, the family $\mathcal{H}_{a,b}\subset Hilb_C^{\PP^3}$ of curves with minimal free resolution (\ref{eq::ACMresolution}) is a smooth open set of an irreducible component of $Hilb_C^{\PP^3}$ by Proposition \ref{prop::Migliore}. With respect to the natural inclusions $\mathcal{H}_C^X\subset Hilb_C^X\subset Hilb_C^{\PP^3}$, we have that $\mathcal{H}_C^X=\mathcal{H}_{a,b}\cap Hilb_C^X$ is open in $Hilb_C^X$. The result follows.
\end{proof}

\medskip
\begin{proof}[Proof of Corollary \ref{candites}]
Let $(a, b)$ be an admissible pair of degree $d$. Consider the incidence correspondence:
    \begin{equation*}
\xymatrix{    & \mathcal{I} \ar[dr]_{\pi} \ar[dl]&\! \! \! \! \! \! \! \! \! \!:=\{(C,X)\in \mathcal{H}_{a,b} \times|\OO_{\PP^3}(d)|:C\subset X\} \\
\mathcal{H}_{a,b} &  & \ \ \ \ \ det(a,b)\subset |\mathcal{O}_{\PP^3}(d)| } 
\end{equation*}
 By Proposition \ref{prop::keyObservation} and Proposition \ref{prop::Migliore}, the projection $\pi$ is dominant onto $det(a,b)$.
 Consider $X\in det(a,b)^{sm}$ and a curve $C\in \mathcal{H}_{a,b}$ contained in $X$ constructed by adding a row. Corollary \ref{cor::detInNL} implies that the line bundle $\mathcal{L}:=\OO_X(C)$ is not a multiple of $\mathcal{O}_X(1)$. That is, $X\in NL(d)$, and the first Chern class $c_1(\mathcal{L})$ of $\mathcal{L}$ is a primitive class in $H^{1,1}(X;\mathbb{Z}).$ 
 
 If $d=4$, then this Corollary follows from Lemma \ref{lem::simplerCurves}. If $d\ge 5$, we know that $det(a,b)^{sm}\subset \Sigma_{\mathcal{L}}$, where $\Sigma_{\mathcal{L}}$ stands for the component  of $NL(d)$ whose elements are surfaces that keep the integral class $c_1(\mathcal{L})$ as a class of type (1,1). Suppose that the containment $det(a,b)^{sm}\subset \Sigma_{\mathcal{L}}$ is strict. Thus, there must exist a 1st-order deformation $X'$ of $X$ and a lifting to it of the line bundle $\mathcal{L}$ which has no sections; otherwise $X'$ is in the image of $\pi$. However, $h^1(X,\mathcal{L})=0$ (Lemma \ref{lem::goodCohomologicalProperties}) implies that the section of $\mathcal{L}$ extends to any 1st-order deformation of the pair $(X,\mathcal{L})$
\cite[Corollary 3.3.15]{sernesi}. Thus any 
 $X'$ contains a curve in $\mathcal{H}_{a,b}$ which means it is in the image of $\pi$; this is a contradiction. By Corollary \ref{cor::detInNL} it is enough to consider 1st-order deformations of $\mathcal{L}$, and therefore, we must have $\Sigma_{\mathcal{L}}=det(a,b)^{sm}$ as claimed.
  \end{proof}

\medskip\noindent
\begin{remark} \label{quinticCurve}
A Noether-Lefschetz component $\Sigma$ may fail to be in the image of the morphism $\pi$ in an incidence correspondence similar to $\mathcal{I}$ above. For example, consider the family of surfaces of degree $5$ that contain a rational quintic curve. If $X$ is in this family and $\mathcal{L}=\mathcal{O}_X(C)$, where $C$ is a rational quintic curve, then the map $\pi:\mathcal{I}\to |\mathcal{O}_{\PP^3}(5)|$ is not dominant onto the family of surfaces with Picard group generated by $\mathcal{O}(1)$ and $\mathcal{L}$. In fact, in this case the image of $\pi$ has dimension $49$ which contradicts Theorem $\star$. Of course, $h^1(X,\mathcal{L})\ne 0$ which makes the proof of Corollary \ref{candites} fail.
\end{remark}

\medskip\noindent
Keeping the notation of Corollary \ref{candites}, let us consider the set of 1st-order deformations of $\mathcal{L}$ that lift to a 1st-order deformation $X'$ of $X$. This is the tangent space to $\Sigma_{\mathcal{L}}\subset NL(d)$ at $X$, which we denote $T_X\Sigma_{\mathcal{L}}$.

\medskip\noindent
Let $\Sigma_{\mathcal{L}}\subset NL(d)$ be a component formed by smooth determinantal surfaces, and $X$ a general point in it. Since $h^1({\mathcal{L}})=0$, then we can compare deformations in $T_X\Sigma_{\mathcal{L}}$ to those of the section of ${\mathcal{L}}$. This is what we do next.

\begin{prop}\label{prop::tangentDet}
    For an admissible pair $(a, b)$, consider a general surface $X\in det(a,b)^{sm}=\Sigma$ and a curve $C\subset X$ obtained by adding a row. Then 
    $$T_X\Sigma \cong\frac{T_C\mathcal{H}_{a,b}}{T_C|\mathcal{O}_X(C)|}.$$
    In particular, $\mathrm{dim}\  \! T_X\Sigma=\mathrm{dim} \ \!\Sigma$.
\end{prop}

\begin{proof} 
The normal-bundle sequence 
$0\to N_{C/X}\to N_{C/\PP^3}\to N_{X/\PP^3}|_C\to 0$
induces the following exact sequence in cohomology:
\begin{equation}\label{delta}
0\to H^0(\mathcal{O}_C(C))\to T_C\mathcal{H}_{a,b}\overset{\pi}{\longrightarrow} T_X\PP^N\overset{\delta}{\longrightarrow}H^1(\mathcal{O}_C(C))
\end{equation}
where $\PP^N=|\mathcal{O}_{\PP^3}(d)|$. The tangent space satisfies $T_X\PP^N\cong H^0( N_{X/\PP^3}|_C)$ because $C$ is ACM and by construction contained in a unique surface $X$ of degree $d$. Now, the image of $\pi$ consists of 1st-order deformations of $X$, as subscheme of $\PP^3$, that contain a 1st-order deformation of $C$. In other words, $T_C\mathcal{H}_{a,b}/ker(\pi)\subset T_X\Sigma$. This inclusion is an equality because the sections of $\mathcal{L}=\mathcal{O}_X(C)$ extend to any of its deformations as a consequence of $H^1(\mathcal{L})=0$ \cite[Prop. 3.3.4]{sernesi}. It follows from the regularity of $X$, Lemma \ref{lem::goodCohomologicalProperties} and the sequence $0\to \mathcal{O}_X\to \mathcal{O}_X(C)\to \mathcal{O}_C(C)\to 0$, that $ker(\pi)\cong T_C|\mathcal{O}_X(C)|$. Since $\mathcal{H}_{a,b}$ is generically smooth then we have that $\mathrm{dim}\  \! T_X\Sigma=\mathrm{dim} \ \!\Sigma$ by Corollary \ref{candites} and Theorem \ref{thm::mainTheorem}. 
\end{proof}

\begin{remark} Even if a curve $L$ is contained in a unique degree $d$ smooth surface $X$, the space $H^0(N_{X/\PP^3}|_C)$ can fail to be in bijection with $T_X\PP^N$; unlike the case in (\ref{delta}). This would make Proposition \ref{prop::tangentDet} also fail.
As an example, let $L$ be a smooth curve residual to a smooth rational quintic curve (considered in Remark \ref{quinticCurve}), in a complete intersection of two smooth surfaces of degree 5 and 6. Here, $L$ is contained in a unique degree 5 surface $X$, but Proposition \ref{prop::tangentDet} fails as $L$ is not ACM and $h^0(N_{X/\PP^3}|_L)=56> \mathrm{dim}\ \! T_X\PP^{55}.$
\end{remark}

\medskip
\begin{proof}[Proof of Corollary 2] The component $det(a,b)^{sm}=\Sigma$ is generically smooth because $\mathrm{dim}\  \! T_X\Sigma=\mathrm{dim} \ \!\Sigma$ for a general $X$. Let us now write the codimension of $\Sigma$ in terms of the curve $C$ constructed in Prop. \ref{prop::keyObservation} by adding a row.
    Since $X$ is regular and $H^1(\mathcal{O}_X(C))=0$, then
    $H^1(\mathcal{O}_C(C))\cong H^2(\mathcal{O}_X)$. If $d=4$, then the result follows from Lemma \ref{lem::simplerCurves}. 
    If $d>4$, then the normal-bundle sequence  yields:
    \begin{equation*}
\xymatrix{  T_X\PP^N \ar[r]^{\delta} \ar@{->>}[d]^{f} & H^1(\mathcal{O}_C(C)) \ar[d]^{\cong} \ar[r]& H^1(N_{C/\PP^3})\ar[r]^g \ar@{=}[d] & H^1(\mathcal{O}_C(d))\ar@{=}[d]\to 0\\
H^1(T_X)\ar[r]^{\cup c(\mathcal{L})} & H^2(\mathcal{O}_X)\ar[r] &H^1(N_{C/\PP^3})\ar[r]^g &H^1(\mathcal{O}_C(d))\to 0}
\end{equation*}
The dimension of the image of the map $\cup c(\mathcal{L})$ is the codimension of $\Sigma$. Indeed, the kernel of $\cup c(\mathcal{L})$ consists of the first-order deformations of $X$, as an abstract variety, along the line bundle $\mathcal{L}$ by \cite[Th. 3.3.11]{sernesi}. The map $f$ sends a first-order deformation of $X$ as a subscheme in $\PP^3$ to a first-order deformation of $X$ as an abstract variety. If $d>4$, then $f$ is onto \cite[Th. 18.5]{Kodaira}. The exact sequence (\ref{delta}) implies that the result follows from the commutativity of this diagram and the surjectivity of the map $g$.
\end{proof}

\begin{table}[t]
\begin{center}
    \begin{tabular}{|c|c|c|}
        \hline $d$ & $\textrm{codim } det(a,b)$ & \# components \\
        \hline \hline
        3 & 0 & 1 \\\hline
        4 & 5:1 & 5 \\ \hline
        5 & 2, 3, 6:4 & 8 \\ \hline
        6 & 3, 5, 6, 7, 8, 2:9, 9:10 & 16 \\ \hline
        7 & 4, 7, 9, 10, 12, 14, 15, 2:16, 17, 2:18, 2:19, 11:20& 25 \\ \hline
        8 & 5, 9, 12, 2:13, 16, 19, 21, 22, 2:23, 24, 26, & 44\\
        & 27, 28, 2:29, 2:30, 3:31, 2:32, 3:33, 3:34, 14:35 & \\
        \hline
        9 & 6, 11, 15, 16, 17, 20, 24, 27, 28, 30, 2:31, 32, 34, 36, 37, 39, 40, 2:41, 42, & 68\\
        & 2:43, 44, 45, 2:46, 47,  2:48, 2:49, 4:50, 2:51, 4:52, 2:53, 4:54, 3:55, 17:56 & \\
        \hline
        \end{tabular} 
    \caption{Codimension of components of families of determinantal surfaces of a degree $d$. The symbol $k:r$ denotes that $r$ occurs as codimension $k$ times. }\label{T2}
\end{center}
\end{table}

\medskip
\section{Distinguished Noether-Lefschetz components}
\noindent 
We start this section by studying linear determinantal surfaces.
Later, we examine the codimension of the components $det(a,b)^{sm}\subset NL(d)$ in terms of the degree $d$ and length $t$ of the admissible pair $(a,b)$. This contrasts with Corollary 2 which relies on the geometry of curves contained in determinantal surfaces. In particular, Corollary \ref{cor::specialComponents} conjectures that $det(a,b)^{sm}$ is a special component of $NL(d)$ if $t\leq\lfloor\frac{d-1}{3}\rfloor$.

\subsection{Linear determinantal surfaces}
As an application of Theorem \ref{thm::mainTheorem}, we compute the dimension of the locus of linear determinantal surfaces of degree $d$, recovering \cite[Theorem 1.3]{reichstein2017dimension} in the case $r=3$ with a different method. This computation confirms that linear determinantal surfaces form a component of the Noether-Lefschetz locus $NL(d)$ of maximal codimension $\binom{d-1}{3}$, which had been established in \cite{CiroAngelo}. Notice that Theorem \ref{thm::mainTheorem} tells us that there are many more components $det(ab)$ of maximal codimension.
\begin{cor}[of Theorem \ref{thm::mainTheorem}]\label{cor::linearEntries}
    The linear determinantal surfaces of degree $d$ form a family of dimension $2d^2+1.$
\end{cor}
\begin{proof}
    Let $t:=d$. The family of linear determinantal surfaces coincides with $det(a,b)$ for the admissible pair $(a,b)$ given by $a_i:=d+1$, $b_j:=d+2$ for all $i,j$. 
    We apply Theorem \ref{thm::mainTheorem} where $a_0=d$ and
    $$A=\OO_{\PP^3}(-(d+1))^d\oplus\OO_{\PP^3}(-d),\quad B=\OO_{\PP^3}(-(d+2))^d.$$
    Following the expression (\ref{eq::mainTheorem}), we verify that $hom(B,B)=d^2$, $hom(A,A)=d^2+4d+1$, and
    \begin{align*}
        hom(B,A)&=d^2\binom{(d+2)-(d+1)+3}{3}+d\binom{(d+2)-d+3}{3}\\
        &=4d^2+10d.
    \end{align*}
    Similarly, using (\ref{eq::degreeAndGenus}) we obtain the expressions $$d_C=\frac{d(d+3)}{2},\quad g_C=\frac{(d-1)(2d^2+5d-6)}{6}.$$
    Putting everything together, we get that $\dim det(a,b)=2d^2+1$.
\end{proof}

\noindent
Let us mention other distinguished components of $NL(d)$ which are realized by determinantal surfaces.
\begin{remark}\label{rmk::lineConic}
    Let $\Sigma_1, \Sigma_2\subset NL(d)$ be the families of surfaces containing a line or a conic, respectively. Both are of the form $det(a,b)^{sm}$, the former for $a=(0,0),b=(1,d-1)$ while the latter for $a=(0,1), b=(2,d-1)$. Consequently, Theorem \ref{thm::mainTheorem} implies the well-known fact $\mathrm{codim}\ \Sigma_1=d-3$, and $\mathrm{codim}\ \Sigma_2 = 2d-7$. 
\end{remark}

\medskip\noindent
M. Green and C. Voisin showed that any component of $NL(d)$, distinct to  $\Sigma_1,\Sigma_2$ above, has strictly larger codimension than $2d-7$ if $d\geq5$. That is, $\Sigma_1,\Sigma_2$ are unique in their codimension \cite{Green2,Voisin2}. On the other hand, Theorem \ref{thm::mainTheorem} tells us there are many distinct special components of the same codimension. Avoiding these cases, it would be interesting to know whether determinantal surfaces form special components of $NL(d)$ which are unique in their codimension.

\begin{remark}\label{rmkCA} C. Ciliberto and A.F. Lopez \cite{CiroAngelo} proved that for any $d\ge 8$ there is a component in $NL(d)$ of codimension $c$ for every integer $c$ such that $$\mathrm{min}\{\tfrac{3}{4}d^2-\tfrac{17}{4}d+\tfrac{19}{3},\tfrac{9}{2}d^{\tfrac{3}{2}}\}\le \ \! c \ \! \le \binom{d-1}{3}.$$
This result leaves the question open whether for $c< \tfrac{9}{2}d^{\tfrac{3}{2}}$ there is a component in $NL(d)$ of codimension $c$ in case $d>46$. On the other hand, for any $d>46$, determinantal surfaces provide many examples of components $det(a,b)^{sm}$ whose codimension is less than $\tfrac{9}{2}d^{\tfrac{3}{2}}$. For example, for $d=47$, in the interval $44\leq c\leq \frac{9}{2}d^{\frac{3}{2}}=1449$, there are 235 different values of $c$ that realize the codimension of a component formed by determinantal surfaces induced by matrices of size $2\leq t\leq 4$. In \cite{MV}, we release Macaulay2 code for the reader to explore such components using a computer.
\end{remark}

\subsection{Combinatorial constrains on the codimension of $det(a,b)$.} Fix $d$ and $t$. We define two admissible pairs $(a_{d,t}^{min},b_{d,t}^{min})$ and $(a_{d,t}^{max},b_{d,t}^{max})$ of degree $d$ and length $t$ and compute the dimension of the families they determine. We expect these dimensions to be the lower and upper bounds for $\dim det(a,b)$, when $(a,b)$ is an admissible pair of degree $d$ and length $t$.

\begin{cor}[of Theorem \ref{thm::mainTheorem}]\label{cor::minMax}
    Given $2\leq t\leq d$, let $(a_{d,t}^{max}, b_{d,t}^{max})$ be the pair given by
    \begin{align*}
        \left(a_{d,t}^{max}\right)_i&:=\left\{
            \begin{array}{cc}
                0 & i=1\\
                d-t & i\geq 2
            \end{array}
        \right.\\
        \left(b_{d,t}^{max}\right)_j&:=d-t+1.
    \end{align*}
    Then
    \begin{align*}
        \dim |\OO_{\PP^3}(d)| - \dim det(a_{d,t}^{max},b_{d,t}^{max})=\left\{
        \begin{array}{cc}
            \frac{t(t-1)(3d-2t-5)}{6}  & t<d\\
            \binom{d-1}{3} & t=d.
        \end{array}
        \right.
    \end{align*}
   Similarly, set $d=tk+r$, with $0\leq r<t$ and define the pair $(a_{d,t}^{min},b_{d,t}^{min})$ by
    \begin{align*}
        \left(a_{d,t}^{min}\right)_i&:=0\\
        \left(b_{d,t}^{min}\right)_j&:=\left\{
            \begin{array}{cc}
                k & j\leq t-r\\
                k+1 & j>t-r.
            \end{array}
        \right.
    \end{align*}
    Then
    $$\textrm{dim}\ det(a_{d,t}^{min},b_{d,t}^{min})=\binom{k-1}{3}t^2+\binom{k-1}{2}rt+2d^2+1.$$
\end{cor}
\begin{proof}
    This is a direct consequence of Theorem \ref{thm::mainTheorem} and we omit the computations as no difficulty arises. However, let us note that to apply Theorem 1 to $(a_{d,t}^{max},b_{d,t}^{max})$, we need to consider the free sheaves
    \begin{align*}
        A&:=\OO_{\PP^3}(-(2d-t+1))^{t-1}\oplus\OO_{\PP^3}(-d-1)\oplus\OO_{\PP^3}(-d),\\
        B&:=\OO_{\PP^3}(-(2d-t+2))^t.
    \end{align*}
    We have that $2d-t+1\geq d+1$ with equality if and only if $t=d$. This explains the necessity of separating the result in two cases: the term $hom(A,A)$ increases precisely when $t=d$.
\end{proof}

\medskip\noindent 
    We have verified the following conjecture for $d\leq 28$. See \cite{MV} for the Macaulay2 code  used to do so.
    \begin{conj}\label{conj::minMaxConj}
        If $(a,b)$ is an admissible pair of degree $d$ and length $t$, then
        $$\mathrm{dim }\ det(a^{min}_{d,t},b^{min}_{d,t})\leq \mathrm{dim }\ det(a,b)\leq\mathrm{dim }\ det(a^{max}_{d,t},b^{max}_{d,t}).$$
    \end{conj}

    \medskip\noindent 
    Conjecture \ref{conj::minMaxConj} relates to Theorem $\star$ as follows. The admissible pair $(a^{min}_{d,d},b^{min}_{d,d})$ determines the family of linear determinantal surfaces, and the pair $(a^{max}_{d,2}, b^{max}_{d,2})$ the family of surfaces containing a line. Furthermore, for any $2\leq t\leq d$ we have that
    $$\mathrm{dim}\ det(a^{min}_{d,d},b^{min}_{d,d})\leq\mathrm{dim}\ det(a^{min}_{d,t},b^{min}_{d,t})\leq\mathrm{dim}\ det(a^{max}_{d,t}, b^{max}_{d,t})\leq\mathrm{dim}\ det(a^{max}_{d,2}, b^{max}_{d,2}).$$
    Therefore, for irreducible components of the Noether-Lefschetz locus arising from determinantal surfaces, Conjecture \ref{conj::minMaxConj} predicts finer bounds than those of Theorem $\star$.  Indeed, for $d=t$ there is only one admissible pair corresponding to linear determinantal surfaces, which form a general component of $NL(d)$. For $t=d-1$ and $t=d-2$, the two bounds above coincide, and this allows us to conclude the following.
    \begin{cor}[of Conjecture \ref{conj::minMaxConj}]
        If $(a,b)$ is an admissible pair of degree $d$ and length $t=d-1$ or $t=d-2$ then $det(a,b)^{sm}$ is a general component of $NL(d)$.
    \end{cor}
    \medskip\noindent On the other hand, Conjecture \ref{conj::minMaxConj} also predicts many special components of $NL(d)$ for low values of $t$.

    \begin{cor}[of Conjecture \ref{conj::minMaxConj}]\label{cor::specialComponents}
        If $(a,b)$ is an admissible pair of degree $d$ and length $t$ with
        $$t\leq\left\lfloor\frac{d-1}{3}\right\rfloor$$
        then $det(a,b)^{sm}$ is a special component of $NL(d)$.
    \end{cor}
    \begin{proof}
        According to Corollary \ref{cor::minMax} and Conjecture \ref{conj::minMaxConj} one has
        \begin{align*}
            dim\ det(a,b)&\geq\binom{k-1}{3}t^2+\binom{k-1}{2}rt+2d^2+1>2d^2+1
        \end{align*}
        if either $k\geq4$, or $k\geq3$ and $r\neq0$. The condition $t\leq\frac{d-1}{3}$ ensures that $k\geq3$, so the only case in which the conclusion could fail is $k=3, r=0$. However, in this case
        $t=\frac{d}{3}>\left\lfloor\frac{d-1}{3}\right\rfloor$, which has been ruled out.
    \end{proof}
    \medskip\noindent

\medskip
\section{Quartic surfaces}
\noindent
We now turn our attention to the family of determinantal surfaces of degree 4. This family consists of 5 irreducible divisors, which we will denote by  $\mathcal{F}_1,\ldots,\mathcal{F}_5\subset|\OO_{\PP^3}(4)|$ depending on the options that an admissible pair $(a,b)$ of degree 4 has. In this section we compute the degree of each of these divisors.

\medskip\noindent
The (non-redundant) options for an admissible pair $(a,b)$ of degree 4 are listed in Table \ref{tab::quarticDivisors}. In this table, we denote by $X_i\in\mathcal{F}_i$, with $1\leq i\leq 5$, a general element and by $C_i\subset X_i$ a curve obtained by removing a column (see Prop. \ref{prop::keyObservation}). This table also lists the degree $d_i$ and the genus $g_i$ of the curve $C_i$. We keep the convention of normalizing an admissible pair $(a,b)$ of degree 4, so that $a_1=5$.

\medskip\noindent 
Combining Theorem \ref{thm::mainTheorem} and Proposition \ref{prop1} we get the following.
\begin{cor}[of Theorem \ref{thm::mainTheorem}]
    The families $\mathcal{F}_i$ are irreducible divisors inside $|\OO_{\PP^3}(4)|$ for $1\leq i\leq5$.
\end{cor}

\medskip\noindent
The following lemma shows that each $\mathcal{F}_i$ coincides with a known family of quartic surfaces.
 \begin{lem}\label{lem::simplerCurves}
    The component $\mathcal{F}_i$ coincides with the family of quartic surfaces containing the following curves
    \begin{enumerate}
        \item[$\mathcal{F}_1:$] a degree 6 and genus 3 curve given by the resolution
        $$0\to\OO_{\PP^3}(-4)^3\to\OO_{\PP^3}(-3)^4.$$
        \item[$\mathcal{F}_2:$] a twisted cubic.
        \item[$\mathcal{F}_3:$] the complete intersection of two quadric surfaces.
        \item[$\mathcal{F}_4:$] a line.
        \item[$\mathcal{F}_5:$] a conic.
    \end{enumerate}
 \end{lem}
 \begin{proof}
     Since all the families in question are irreducible of codimension 1 in $|\OO_{\PP^3}(4)|$, it is enough to note that the curve $C_i$ obtained by removing a column is, in each case, the listed curve.
 \end{proof}

\begin{table}[t]
\begin{center}
    \begin{tabular}{|c|c|c|c|c|c|}
        \hline $\mathcal{F}_i$ & $a$ & $b$ & $d_i$ & $g_i$ & $\Delta_i$\\
        \hline \hline
        $\mathcal{F}_1$ & (5,5,5,5) & (6,6,6,6) & 6 & 3 & 20\\
        $\mathcal{F}_2$ & (5, 5, 5) & (6, 6, 7) & 3 & 0 & 17\\
        $\mathcal{F}_3$ & (5,5) & (7,7) & 4 & 1 & 16\\
        $\mathcal{F}_4$ & (5,5) & (6,8) & 1 & 0 & 9\\
        $\mathcal{F}_5$ & (5,6) & (7,8) & 2 & 0 & 12\\
        \hline
    \end{tabular} \centering \caption{The divisor $\mathcal{F}_i=det(a,b)\subset|\OO_{\PP^3}(4)|$ parametrizes quartic surfaces containing a curve of degree $d_i$ and genus $g_i$, and whose Picard lattice has discriminant $\Delta_i.$}\label{tab::quarticDivisors}
\end{center}
\end{table}
 
 \medskip\noindent 
It follows from the previous Lemma and Corollary \ref{cor::detInNL} that the intersection matrix of $Pic(X_i)$, with respect to the hyperplane section class $H$ and the class $[C_i]$, has the form
\begin{align*}
    A_i=\left(
    \begin{array}{cc}
        4 & d_i\\
        d_i & 2g_i-2
    \end{array}
    \right).
\end{align*}
Another important invariant which is listed in Table \ref{tab::quarticDivisors} is the \textit{discriminant} of $Pic(X)$, which we denote by $\Delta_i:=-det(A_i)$.

\medskip\noindent We are interested in computing the degree of each $\mathcal{F}_i$ inside $|\OO_{\PP^3}(4)|\cong\PP^{34}$. In order to do this, let us recall briefly the Noether-Lefschetz numbers for $K3$ surfaces. See \cite{MP} for an exposition on the topic.
\subsection{Noether-Lefschetz numbers} We consider the moduli space $\mathcal{M}_4$ of quasi-polarized K3 surfaces of degree 4. That is, $\mathcal{M}_4$ parametrizes K3 surfaces $X$, up to isomorphism, together with a divisor class $H$ of self-intersection $H^2=4$. The space $\mathcal{M}_4$ is a 19-dimensional variety and the general element $(X,H)\in\mathcal{M}_4$ has Picard group
$$Pic(X)\cong\mathbb{Z}\cdot H.$$
Noether-Lefschetz loci arise when one considers surfaces with Picard rank 2. Following \cite{MP}, we consider two types of Noether-Lefschetz loci. The first type corresponds to specifying a rank 2 lattice and asking that $Pic(X)$ is isomorphic to this lattice, with $H$ one of its generators. To specify such a lattice, we can assume that its intersection matrix has the form
\begin{align*}
    \mathbb{L}_{h,d}:=\left(
    \begin{array}{cc}
        4 & d\\
        d & 2h-2
    \end{array}
    \right).
\end{align*}
Up to isomorphism, the lattice is determined by two invariants. Namely, the \textit{discriminant} 
$$\Delta:=-det(\mathbb{L}_{h,d})=d^2-8h+8$$
and the \textit{coset} 
$$\delta\equiv d\mod 4.$$
One then may consider the integral divisor in $\mathcal{M}_4$ supported in the closure of all surfaces having Picard lattice of rank 2 with discriminant $\Delta$ and coset $\delta$
$$P_{\Delta,\delta}\subset\mathcal{M}_4.$$
Observe that, according to the Hodge index Theorem, $P_{\Delta,\delta}=0$ if $\Delta\leq0$. Also, $P_{\Delta,\delta}=0$ unless $\Delta-\delta^2$ is a multiple of 8 due to the fact that
    $$\Delta=\Delta(h,d):=d^2-8h+8\equiv \delta^2\mod 8.$$
    In this case, one can find a lattice $\mathbb{L}_{h,d}$ with invariants $(h,d)$ by choosing $d=\delta$ and $h=\frac{\Delta-\delta^2}{8}-1.$

\medskip\noindent
The second type of Noether-Lefschetz divisors is defined as
\begin{equation}\label{eq::noetherLefschetzDivisors}
    D_{h,d}:=\sum_{\Delta,\delta}\mu(h,d|\Delta,\delta)\cdot P_{\Delta,\delta}
\end{equation}
where $\mu(h,d|\Delta,\delta)$ is the number of classes $D$ in the lattice with invariants $\Delta,\delta$ such that $D^2=2h-2$ and $H\cdot D = d$. This number is always 0, 1 or 2. If $\mu(h,d|\Delta,\delta)\neq0$ then the classes $H, D$ do not necessarily generate the whole lattice with invariants $\Delta,\delta$; they may generate a sublattice isomorphic to $\mathbb{L}_{h,d}$. In this case, the corresponding discriminant $\Delta(h,d)$ is divisible by $\Delta$ and the quotient $\Delta(h,d)/\Delta$ must be a perfect square. If $\Delta=\Delta(h,d)$, then $D$ and $H$ do generate the whole lattice, hence $d\equiv\delta\mod4$.

\medskip\noindent Now consider a general pencil $\pi_0:\PP^1\to|\OO_{\PP^3}(4)|\cong\PP^{34}$. There is a dominant rational map $|\OO_{\PP^3}(4)|\dashrightarrow\mathcal{M}_4$ sending a quartic surface $X=\{F=0\}\subset\PP^3$ to its isomorphism class $X$, together with the hyperplane class $H\subset X$. The composition of this map and $\pi_0$ will be denoted by $\pi$.
\begin{definition}
    The \textit{Noether-Lefschetz number} $NL_{h,d}$ is the degree of $\pi^*(D_{h,d})$.
\end{definition}
\begin{prop}\cite[Corollary 2]{MP}\label{prop::modularForm}
    There exists an effectively computable modular form $\Theta-\Psi$ such that, if $\Delta(h,d)>0$, then $NL_{h,d}$ is equal to the coefficient of $q^{\Delta(h,d)/8}$ in $\Theta-\Psi$.
\end{prop}

\noindent We do not write down the exact definition of $\Theta-\Psi$ as we are only interested in the cases $\Delta\leq 20$. The following partial expression contains all the information we care about:
$$\Theta-\Psi=-1+320q^{9/8}+5016q^{3/2}+76950q^2+136512q^{17/8}+640224q^{5/2}+\ldots$$

\medskip\noindent
The reason we mentioned two descriptions of Noether-Lefschetz loci is that while the Noether-Lefschetz numbers $NL_{h,d}$ can be directly read off from the modular form $\Theta-\Psi$, we are actually interested in the analogous numbers defined by the classes $P_{\Delta,\delta}$; namely, $deg(\pi^*P_{\Delta,\delta})$. Therefore, we need to express $P_{\Delta,\delta}$ in terms of some $D_{h,d}$'s using (\ref{eq::noetherLefschetzDivisors}). This is what we do next.

\noindent
\subsection{Degrees of the divisors $\mathcal{F}_i$.} 
As a consequence of Lemma \ref{lem::simplerCurves} we have that 
$$deg(\mathcal{F}_i)=deg(\pi^*P_{\Delta_i,\delta_i}).$$
In order to use the modular form $\Theta-\Psi$ above in computing $deg(\mathcal{F}_i)$, we first need to write $P_{\Delta,\delta}$ in terms of divisors $D_{h,d}$. We will work out explicitly the case of $\mathcal{F}_3$ and list the results for the other cases. 

\medskip\noindent The first step is to consider the unique divisor $D_{h,d}$ for which a lattice $\mathbb{L}_{h,d}$ has invariants $\Delta_3=16$ and $\delta_3\equiv d_3\mod 4$. Since $d_3=16$, then $\delta\equiv 0 \mod 4$. This divisor is $D_{31,16}$ and is determined by the curve $C_3$. We now want to express $D_{31,16}$ in terms of various divisors $P_{\Delta,\delta}$. That is, we must find all the possible pairs $(\Delta,\delta)$ for which the lattice $\mathbb{L}_{31,16}$ has a rank 2 sub lattice with invariants $\Delta$ and $\delta$. In order to do this, we consider the discriminants $\Delta$ such that $\Delta_3/\Delta$ is a perfect square, and the possible cosets $\delta$ such that $\Delta\equiv\delta^2\mod 8$. We deduce that $D_{31,16}$ is a linear combination of $P_{16,0}, P_{4,2}, P_{1,1}$ and $P_{1,3}.$

\medskip\noindent The second step is to compute the multiplicities $\mu(31,16|16,0),\mu(31,16|4,2),\mu(31,16|1,1)$ and $\mu(31,16|1,3)$. We work out the first one. We need to consider a lattice $\mathbb{L}_{31,16}$, generated by classes $C,H$, and compute the number of classes $D=aH+bC$ such that
\begin{align*}
    60&=D^2=4a^2+32ab+60b^2\\
    16&=H\cdot D=4a+16b.
\end{align*}
That is, such that $D$ and $H$ span a sublattice of rank 2 with the given invariants. In this case, one solution is $D=C$. There is a second solution, namely $D=8H-C$. Therefore, $\mu(31,16|16,0)=2$. Similarly, one computes the other multiplicities:
\begin{align*}
    \mu(31,16|4,2)&=2\\
    \mu(31,16|1,1)&=2\\
    \mu(31,16|1,3)&=2
\end{align*}
which allows us to write the relation
$$D_{31,16}=2P_{16,0}+2P_{4,2}+2P_{1,1}+2P_{1,3}.$$
Since we are interested in $P_{16,0}$, we are left to repeat this process for the other divisors $P_{4,2}, P_{1,1}$ and $P_{1,3}$ occurring in the expression above. Doing this we obtain similar expressions
\begin{align*}
    D_{1,2}&=2P_{4,2}+2P_{1,1}+2P_{1,3}\\
    D_{1,1}&=P_{1,1}\\
    D_{2,3}&=P_{1,3}.
\end{align*}
Consequently, we get the identity $P_{16,0}=\frac{1}{2}\left(D_{31,16}-D_{1,2}\right)$. To finish up the computation of $deg(\mathcal{F}_3)$, we use Proposition \ref{prop::modularForm}:
\begin{align*}
    deg(\mathcal{F}_3)&=deg\ \pi^*\left(\frac{1}{2}(D_{31,16}-D_{1,2})\right)\\
    &=\frac{1}{2}\left(NL_{31,16}-NL_{1,2}\right)\\
    &=\frac{1}{2}(76950-0)\\
    &=38475.
\end{align*}

\medskip\noindent
Similar computations yield the following identities for the other cases $\mathcal{F}_i$, with $i=1,2,4,5$
   $$ P_{20,2}=\tfrac{1}{2}D_{23,14},\qquad P_{17,1}=D_{35,17},$$
  $$ P_{9,1}=D_{36,17}-D_{1,1}-D_{2,3},\qquad P_{12,2}=\tfrac{1}{2}D_{40,18}.$$

\medskip\noindent
These computations allow us to conclude the following. \setcounter{thmIntro}{1}
\begin{thmIntro}
    The family of determinantal quartic surfaces consists of 5 prime divisors $\mathcal{F}_1,\ldots,\mathcal{F}_5\subset|\OO_{\PP^3}(4)|$ depending on the choice of an admissible pair $(a,b)$ of degree 4. Each of these divisors has degree
    \begin{align*}
        deg(\mathcal{F}_1)&=320112,\\
        deg(\mathcal{F}_2)&=136512,\\
        deg(\mathcal{F}_3)&=38475,\\
        deg(\mathcal{F}_4)&=320,\\
        deg(\mathcal{F}_5)&=2508.
    \end{align*}
\end{thmIntro}

\medskip
\begin{rmk}
It was pointed out to us by Bernd Sturmfels that the degree $\mathcal{F}_3=38475$ can be thought of (and computed) as the degree of the component of forms that are sums of four squares of quadrics inside $\Sigma_{4,4}$, the cone of quartics which are sums of squares \cite[Theorem 1]{Bernd}.
\end{rmk}

\end{document}